\documentclass[11pt, reqno]{amsart}
\usepackage{amsfonts,amssymb,amscd,amsmath,enumerate,verbatim,calc,amsthm}
\usepackage{amsmath}
\usepackage{amssymb}
\usepackage{amscd}
\usepackage{color}
\usepackage{mathtools}
\usepackage[colorlinks,pagebackref=true]{hyperref}

\makeatletter

\def\widebreve#1{\mathop{\vbox{\m@th\ialign{##\crcr\noalign{\kern3\p@}%
				\brevefill\crcr\noalign{\kern3\p@\nointerlineskip}%
				$\hfil\displaystyle{#1}\hfil$\crcr}}}\limits}

\def\brevefill{$\m@th \setbox\z@\hbox{$\braceld$}%
	\bracelu\leaders\vrule \@height\ht\z@ \@depth\z@\hfill\braceru$}
\makeatletter

\def\@citecolor{blue}
\def\@linkcolor{blue}
\def\@urlcolor{blue}
\def\@urlcolor{blue}

\def\NZQ{\mathbb}               
\def\NN{{\NZQ N}}
\def\QQ{{\NZQ Q}}
\def\ZZ{{\NZQ Z}}
\def\RR{{\NZQ R}}

\def\mfp{\mathfrak p}

\def \mm{\mathfrak {m}}
\def\Ass{\operatorname{Ass}}
\def\Min{\operatorname{Min}}
\def\sat{\operatorname{sat}}
\def\height{\operatorname{height}}

\def\liminf{\operatorname {liminf}}
\def\limsup{\operatorname {limsup}}

\newtheorem{Theorem}{Theorem}[section]

\newtheorem{Corollary}[Theorem]{Corollary}

\newtheorem{Remark}[Theorem]{Remark}

\newtheorem{Example}[Theorem]{Example}

\newtheorem{Definition}[Theorem]{Definition}

\newtheorem{Question}[Theorem]{Question}
\let\epsilon\varepsilon
\let\phi=\varphi
\let\kappa=\varkappa

\begin{document}
	\title{ Epsilon multiplicity, multiplicity=volume formula and analytic spread of family of ideals}
	\author{Parangama Sarkar}
	\address{Parangama Sarkar,  Department of Mathematics,
		Indian Institute of Technology, Palakkad, India}
	\email{parangamasarkar@gmail.com, parangama@iitpkd.ac.in}
	\begin{abstract}  In an analytically unramified local ring $(R,\mm)$ of dimension $d\geq 1$, for a  filtration of ideals $\mathcal I=\{I_m\}_{m\in\NN}$ satisfying $\mathcal A(r)$ condition and  for any $\mm$-primary ideal $K$, it is shown  in \cite{Pa2025} that the epsilon multiplicity of the weakly graded family of ideals  $\{(I_m:K)\}_{m\in\NN}$ exists as a limit and it is bounded above by the epsilon multiplicity of $\mathcal I$, $\epsilon(\mathcal I)$. In this article, we first show that $\epsilon(\mathcal I)$ coincides with the epsilon multiplicity of $\{(I_m:K)\}_{m\in\NN}$ and this leads to the following: $(a)$ an expression for $\epsilon(\mathcal I)$ as a limit of the epsilon multiplicities of other graded families of ideals and $(b)$ a multiplicity=volume formula for the epsilon multiplicity of an ideal $I$ in $R$. In the final part of the article, we investigate the maximality of the analytic spread of filtrations of ideals.
\end{abstract}
\maketitle
\section{Introduction}
Let $(R,\mm)$ be a Noetherian local ring of dimension $d$ and $I$ be an ideal in $R.$ 
The epsilon multiplicity of an ideal $I$ in a local ring $R$ is defined as
$$
\epsilon(I)=d!\limsup_{m\to\infty}\frac{\lambda_R(H^0_{\mm}(R/I^m))}{m^d}
$$ (here the length of an $R$-module $M$ is denoted by $\lambda_R(M)$) and it is a finite real number \cite{UV}, \cite{KV}. If $I$ is $\mm$-primary then for all large $m,$  the function $\lambda_R(H^0_{\mm}(R/I^m))=\lambda_R({R}/{I^m})$ coincides with a polynomial in $m$ of degree $d$  and hence $\epsilon(I)=e(I)$ is a limit and a positive integer where $e(I)$ is the Hilbert-Samuel multiplicity of $I$ \cite{hilb}, \cite{samuel}. The question of when the epsilon multiplicity exists as a limit has been investigated by several mathematicians \cite{CHST05}, \cite{HPV}, \cite{CHS}. In 2014, Cutkosky proved this in greater generality. 
\begin{Theorem}{\em(\cite[Corollary 6.3]{C1})}\label{TheoremI1} Suppose that $I$ is an ideal in an analytically unramified local ring $R$. Then $\epsilon(I)$ is actually a limit, 
	$$
	\epsilon(I)
	=d!\lim_{m\rightarrow \infty}\frac{\lambda_R(H^0_{\mm}(R/I^m))}{m^d}.
	$$ 
\end{Theorem}
Moreover, in \cite{CHST05}, Cutkosky et al. showed that this limit can be irrational. However, for monomial ideals and determinantal ideals, the  epsilon multiplicity is known to be a rational number \cite{HPV}, \cite{JMV}. Recently, Cutkosky and Landsittel provided a ``multiplicity=volume" formula for the epsilon multiplicity of an ideal in terms of the Amao multiplicities.
\begin{Theorem}{\rm{(\cite[Theorem 1.1]{CL2024})}}\label{Th2}
Let $I$ be an ideal in an analytically unramified local ring. Then  $$\epsilon(I)=\lim\limits_{n\to \infty}\frac{1}{n^d}\Big(\lim\limits_{m\to \infty}\frac{\lambda_R(((I^n)^{\sat})^m/(I^n)^m)}{m^d/d!}\Big).$$
\end{Theorem}
The epsilon multiplicity of a family  $\mathcal I=\{I_m\}_{m\in\NN}$ of ideals is defined as \begin{equation*}
	\epsilon(\mathcal I):=d!\limsup_{m\to\infty}\frac{\lambda_R(H^0_{\mm}(R/I_m))}{m^d}.
\end{equation*} 
The notion of epsilon multiplicity for a filtration $\mathcal I=\{I_m\}_{m\in\NN}$ (see Definition \ref{fil}) is developed extensively in \cite{CS2024}.
\begin{Definition}{\em	Let $\mathcal I=\{I_m\}_{m\in\NN}$ be a filtration on a local ring $R$ and $r\in \ZZ_{>0}$. We say that $\mathcal I$ satisfies property $\mathcal A(r)$ if 
		$$
		I_m^{\sat}\cap \mm^{rm}=I_m\cap \mm^{rm}  \mbox{ for all $m\in \NN$.}
		$$}
\end{Definition} 
Any discrete valued filtration (see Definition \ref{dis}) in a Noetherian local domain satisfies $\mathcal A(r)$ for some $r\in\ZZ_{>0}$ (\cite[Theorem 3.1]{CS2024}). 
As a consequence of  \cite[Theorem 6.1]{C1}, we get, if $R$ is an analytically unramified  local ring of dimension $d$, and $\mathcal I=\{I_m\}_{m\in\NN}$ is a filtration of ideals in $R$ which satisfies the property $A(r)$ for some $r\in \ZZ_{>0}$, then 
\begin{equation*}\label{eqI6}
	\epsilon(\mathcal I)=d!\lim_{m\rightarrow \infty}\frac{\lambda_R(H^0_{\mm}(R/I_m))}{m^d}
\end{equation*}
is a real number. 
That is, the epsilon multiplicity of $\mathcal I$ exists as a limit. For more details about $\mathcal A(r)$ condition, see \cite{CS2024}. The notion of epsilon multiplicity has been extended to weakly graded family of ideals in \cite{Pa2025}. Let $\mathcal I=\{I_m\}_{m\in\NN}$ be a filtration of ideals in an analytically unramified  local ring $(R,\mm)$ of dimension $d\geq 1$ and $\mathcal I$ satisfy $\mathcal A(r)$ for some $r\in \ZZ_{>0}$. Consider the weakly graded family of ideals (see Definition \ref{wg}) of the form  $\{(I_m:K)\}_{m\in\NN}$ where $K$ is any $\mm$-primary ideal. Recently, in \cite[Theorem 4.4]{Pa2025}, the author proved  the following.
\begin{Theorem}{\em(\cite[Theorem 4.4]{Pa2025})}\label{Th3}
	With $R$, $\mathcal I$ and $K$ as above, 	the epsilon multiplicity of  $\{(I_m:K)\}_{m\in\NN}$ exists as a limit, i.e., $$\displaystyle\lim\limits_{m\to \infty}d!\frac{\lambda\big(H^0_{\mm}(R/(I_m:K))\big)}{m^d}$$ exists and it is bounded above by $\epsilon(\mathcal I)$. 
\end{Theorem}	
	The above theorems, Theorem \ref{Th2} and Theorem \ref{Th3}, motivate the main objectives of this article, which are as follows:
\\{\em {Suppose $(R,\mm)$ is an analytically unramified  local ring of dimension $d\geq 1$.
\begin{enumerate}
	\item Let $\mathcal I=\{I_m\}_{m\in\NN}$ be a filtration of ideals in $R$ that satisfies $\mathcal A(r)$ for some $r\in\ZZ_{>0}$ and $K$ be any $\mm$-primary ideal in $R$. 
	\\Does equality hold in Theorem \ref{Th3}, i.e., is
	$\epsilon(\mathcal I)=\displaystyle\lim\limits_{m\to \infty}d!\frac{\lambda\big(H^0_{\mm}(R/(I_m:K))\big)}{m^d}?$\vspace{2mm}
	\item Can we obtain a ``multiplicity=volume" formula for the epsilon multiplicity of an ideal in $R$, i.e., can the epsilon multiplicity of an ideal be expressed as a limit of the epsilon multiplicities of other suitable ideals?
\end{enumerate}	}}
In this article, we provide affirmative answers to both of the above questions and prove the following.
\begin{Theorem}{\label{TheoremA}}
	Let $(R,\mm)$ be an analytically unramified  local ring of dimension $d\geq 1$ and $K$ be any $\mm$-primary ideal in $R$. 
	\begin{enumerate}
		\item[$(1)$]  Let $\mathcal I=\{I_m\}_{m\in\NN}$ be a filtration of ideals and $\mathcal I$ satisfy $\mathcal A(r)$ condition for some $r\in\ZZ_{>0}$. Then 
		\begin{enumerate} 
			\item[$(i)$] $\epsilon(\mathcal I)=\displaystyle\lim\limits_{n\to \infty}d!\frac{\lambda\big(H^0_{\mm}(R/(I_n:K))\big)}{n^d}.$\vspace{2mm}
			\item [$(ii)$] $\epsilon(\mathcal F(n))$ is a limit  where 
			$\mathcal F(n)=\{F(n)_m=(I_{nm}:K^m) \}_{m\in\NN}$ is a graded family of ideals for all $n\in\ZZ_{>0}$ and \vspace{1mm} $$\epsilon(\mathcal I)=\displaystyle\lim\limits_{n\to\infty}\frac{\epsilon(\mathcal F(n))}{n^d}.$$
		\end{enumerate} 
		\item[$(2)$] Let $\mathcal I=\{I^m\}_{m\in\NN}$ for some ideal $I$ in $R$.  For each $n\in\ZZ_{>0}$, let $J_n$ be an ideal in $R$ such that $I^n\subseteq J_n\subseteq (I^n:K)$.
		Then
		$$\epsilon (I)=\displaystyle\lim\limits_{n\to \infty}\frac{\epsilon (J_n)}{n^d}.$$
	\end{enumerate}
\end{Theorem}	
\noindent
We provide an example of a Noetherian filtration $\mathcal I$ such that all the graded families $\mathcal F(n)$ are non-Notherian for all $n\in\ZZ_{>0}$ (see Example \ref{ex1}).
\\The analytic spread of an ideal $I$ in a (Noetherian) local ring $R$ is defined to be
$\ell(I)=\dim R[I]/\mm R[I]$
where $R[I]=\bigoplus_{m\in\NN}I^m$ is the Rees algebra of $I$ and  $\height(I)\le\ell(I)\le \dim R$ \cite[page 115]{Li}, \cite[Corollary 8.3.9]{HS}. The notion of  analytic spread  of an ideal  have been extended for filtrations of ideals in  \cite{CS2022}. 
Let $\mathcal I=\{I_m\}_{m\in\NN}$ be a graded family of ideals in $R$. The Rees ring of $\mathcal I$ is defined as $R[\mathcal I]=\oplus_{m\in\NN}I_m$ and  the analytic spread of $\mathcal I$ is defined as $\ell(\mathcal I):=\dim R[\mathcal I]/\mm R[\mathcal I]$ and it is bounded above by $\dim R$ \cite[Definition 3.2, Lemma 3.4 and Lemma 3.5]{CS2022} (in this reference the results are proved for filtrations of ideals. But the exact same proofs will work for graded families of ideals).
\\Katz and Validashti studied the positivity of the epsilon multiplicity of an ideal in terms of the maximality of the analytic spread of that ideal  \cite[Theorem 4.7]{KV}.
In \cite[Theorem 4.1]{CS2024}, the authors extended the result in \cite[Theorem 4.7]{KV} for $\QQ$-divisorial filtration of ideals  when $R$ is  a $d$-dimensional excellent normal local domain of equicharacteristic zero, or an arbitrary excellent local domain of dimension $\le 3$ and proved that the epsilon multiplicity of a $\QQ$-divisorial filtration is positive if and only if the analytic spread of the filtration is equal to $\dim R$.
\\Let $I$ be an ideal of maximal analytic spread. This property, however, need not be preserved under taking a colon by an $\mm$-primary ideal, i.e., the analytic spread of $(I:K)$ need not be maximal for any $\mm$-primary ideal $K$. 
\\For example, let $k[x,y]$ be a polynomial ring over a field $k$ and $R=k[x,y]_{({x,y})}$. Consider the ideal $I=(x^2,xy)R$. Then the analytic spread of $I$ is $2$ while the analytic spread of $(I:\mm)$ is $1$ where $\mm=(x,y)R$. 
\\This raises the following natural question:
\begin{Question} Suppose $I$ is an ideal of maximal analytic spread  in $R$. Does it follow that the ideals $(I^n:K)$ have  maiximal analytic spread for all $n\gg 0$ and for any $\mm$-primary ideal $K$?
\end{Question}
As a consequence of Theorem \ref{TheoremA}, we obtain an affirmative answer to this question and show the following. 
\begin{Corollary}
	Let $(R,\mm)$ be an analytically unramified  local ring of dimension $d\geq 1$, $I$ be an ideal in $R$ and $K$ be any $\mm$-primary ideal in $R$. For each $n\in\ZZ_{>0}$, let $J_n$ be an ideal in $R$ such that $I^n\subseteq J_n\subseteq (I^n:K)$. Then the following are equivalent.
	\begin{enumerate}
		\item[$(i)$] $\ell(I)=d$.
		\item[$(ii)$]  $\ell(J_n)=d$ for all $n\gg0$. 
		\item[$(iii)$]  $\ell(J_n)=d$ for some $n\geq 1$.
	\end{enumerate} 
\end{Corollary}		
We conclude the article with the following result on the maximality of the analytic spread of divisorial filtrations of ideals.
\begin{Theorem}
	Let $(R,\mm)$ be an analytically unramified  local ring of dimension $d\geq 1$ and $K$ be any $\mm$-primary ideal in $R$. Let $\mathcal I=\{I_m\}_{m\in\NN}$ be a filtration of ideals in $R$. Consider the graded families of ideals $\mathcal F(n)=\{F(n)_m=(I_{nm}:K^m)\}_{m\in\NN}$  for all integers $n\geq 1$.
	\begin{enumerate}
		\item[$(1)$] Supppose $\mathcal I$ satisfies $\mathcal A(r)$ condition for some $r\in\ZZ_{>0}$. Then the following are equivalent.	\begin{enumerate}
			\item[$(i)$] $\epsilon(\mathcal I)>0$,
			\item[$(ii)$] $\epsilon(\mathcal F(n))>0$ for all $n\gg0$,
			\item[$(iii)$] $\epsilon(\mathcal F(n))>0$ for some $n\geq 1$.
		\end{enumerate}
		\item[$(2)$] If $R$ is an
		excellent normal local domain of dimension $d\geq 1$ of equicharacteristic zero, or an arbitrary excellent local domain with $d\leq 3$ and $\mathcal I$ is a $\QQ$-divisorial filtration of ideals in $R$ then  the following are equivalent.
		\begin{enumerate}
			\item[$(i)$] $\ell(\mathcal I)=d$,
			\item[$(ii)$] $\ell(\mathcal F(n))=d$ for all $n\gg0$,
			\item[$(iii)$] $\ell(\mathcal F(n))=d$ for some $n\geq 1$.
		\end{enumerate}	
			\end{enumerate}	
\end{Theorem}	
\section{Notation and definitions }
We denote the set of nonnegative integers by $\NN$,  the set of positive real numbers, the set of  positive integers and the set of positive rational numbers by $\RR_{>0}$, $\ZZ_{>0}$ and $\QQ_{>0}$ respectively. Throughout this article, $(R,\mm)$ is a Noetherian local ring of dimension $d$. Let $\mathcal I=\{I_m\}_{m\in\NN}$ be a family of ideals in $R$.
\begin{Definition}\label{fil} {\em A graded family $\mathcal I=\{I_m\}_{m\in\NN}$ of ideals in a ring $R$ is a collection of ideals in $R$ such that $I_0=R$ and $I_mI_n\subset I_{m+n}$ for all $m,n\in \NN$. 
	\\A graded family of ideals $\mathcal I=\{I_m\}_{m\in\NN}$ in $R$ is called a filtration if $I_n\subset I_m$ for all $m,n\in \NN$ with $m\leq n$.}
\end{Definition}
A graded family of ideals $\mathcal I=\{I_m\}_{m\in\NN}$ in $R$ is called Noetherian if the Rees algebra $R[\mathcal I]=\bigoplus_{m\in\NN}I_m$ is a finitely generated $R$-algebra. Otherwise we say $\mathcal I$ is non-Noetherian.\vspace{2mm}
\\We denote the $\mm$-adic completion of $R$ by $\hat{R}$ and the set $R\setminus \bigcup_{P\in\Min R}P$ by $R^o$. 
\begin{Remark}\label{nzd1}
	{\em Since ${\hat{R}}$  is a flat $R$-algebra, by \cite[Theorem 9.5]{Mat}, the going-down theorem holds and hence contraction of any minimal prime of $\hat{R}$ is a minimal prime of $R$. Thus for any $c\in R^o$, we have $c\in {\hat{R}}^o$.}
 \end{Remark}	
	\begin{Definition}\label{wg}{\rm\cite{DM} A family of ideals $\mathcal I=\{I_m\}_{m\in\NN}$  in $R$ is called a weakly graded family of ideals  if $I_0=R$ and there exists an element $c\in R^o$ such that $cI_mI_n\subset I_{m+n}$ for all $m,n\geq 1$. }
\end{Definition}
Note that a graded family of ideals is a weakly graded family of ideals (take $c=1$). Suppose $\{J_m\}$ is a weakly graded family  of ideals in $R$ and $c\in R^0$ such that $cJ_nJ_m\subset J_{n+m}$ for all $n,m\in\NN$. Let $K$ be any ideal in $R$ such that $K\cap R^o\neq\emptyset$. Let $d\in K\cap R^o$. Then $cd I_nI_m\subset I_{n+m}$ for all $m,n\in\NN$ where $\mathcal I=\{I_m=(J_{m}:K)\}$.
In this article, we mainly focus on weakly graded families of ideals of the form $\{(I_m:K)\}$ where $\mathcal I=\{I_m\}$ is a filtration of ideals and satisfies $\mathcal A(r)$ condition for some $r\in\ZZ_{>0}$.
For more details about weakly graded families of ideals, see \cite{DM}, \cite{Pa2025}.
\\Let $(R,\mm)$ be a Noetherian local domain of dimension $d$ with quotient field $K$. Let $\nu$ be a discrete valuation of $K$ with valuation ring $\mathcal O_{\nu}$ and maximal ideal $m_{\nu}$. Suppose that $R\subset \mathcal O_{\nu}$. Then for all $m\in \NN$, the valuation ideals are defined as
$$
I(\nu)_m=\{f\in R\mid \nu(f)\ge m\}=m_{\nu}^m\cap R.
$$
Note that $I(\nu)_m$ are $m_{\nu}\cap R$-primary ideals. For any $x\in\RR$, the smallest integer that is greater than or equal to $x$  is denoted by $\lceil x\rceil$.
\begin{Definition}\label{dis} {\em A discrete valued filtration of $R$ is a filtration $\mathcal I=\{I_m\}$ such that there exist discrete valuations $\nu_1,\ldots,\nu_r$ of $K$ and $a_1,\ldots, a_r\in \RR_{>0}$ such that $R\subset \mathcal O_{{\nu}_i}$ for all $1\leq i\leq r$ and for all $m\in \NN$,}
	$$I_m=I(\nu_1)_{\lceil ma_1\rceil}\cap\cdots\cap I(\nu_r)_{\lceil ma_r\rceil}.$$ \end{Definition} 
A divisorial valuation of $R$ (\cite[Definition 9.3.1]{HS}) is a valuation $\nu$ of $K$ such that if $\mathcal O_{\nu}$ is the valuation ring of $\nu$ with maximal ideal $\mathfrak m_{\nu}$, then $R\subset O_{\nu}$ and if $\mfp=\mathfrak m_{\nu}\cap R$ then $\mbox{trdeg}_{\kappa(\mfp)}\kappa(\nu)={\rm ht}(\mfp)-1$, where $\kappa(\mfp)$ is the residue field of $R_{\mfp}$ and $\kappa(\nu)$ is the residue field of $O_{\nu}$. Every divisorial valuation $\nu$ is a  discrete valuation  \cite[Theorem 9.3.2]{HS}. 
\begin{Definition} {\em A divisorial filtration of $R$ is a discrete valued filtration 
	$$
	\{I_m=I(\nu_1)_{\lceil ma_1\rceil}\cap\cdots\cap I(\nu_r)_{\lceil ma_r\rceil}\}
	$$
	where all the discrete valuations $\nu_1,\ldots,\nu_r$ are divisorial valuations.  A divisorial filtration is called $\QQ$-divisorial if $a_i\in \QQ_{>0}$ for all $1\leq i\leq r$. (For  more details about $\QQ$-divisorial filtration, see \cite{C2021}, \cite{CS}.)}
\end{Definition} 
If $I\subset R$ is an ideal in $R$, we define
$I^{\rm sat}=I:m_R^{\infty}=\cup_{n=1}^{\infty}I:m_R^n$.
\section{Main results} In this section, we prove the main results of this article. We begin by noting some properties of a weakly graded family of ideals of the form $\{(I_m:K)\}_{m\in\NN}$ where $\{I_m\}_{m\in\NN}$ is a graded family of ideals and $K$ is an ideal in a local ring $(R,\mm)$.
	\begin{Remark} \label{nzd}{\em 
		  Let $(R,\mm)$ be a Noetherian local ring. 	\begin{enumerate} 
				\item For any ideal $I$ in $R$ and for any $\mm$-primary ideal $K$ in $R$, $(I:K)^{\sat}=I^{\sat}.$
			\item Let $\mathcal I=\{I_m\}_{m\in\NN}$ be a graded family of ideals. Consider a weakly graded family of ideals $\{(I_m:K)\}_{m\in\NN}$ in $R$ where $K$ is an ideal in $R$ with  $K\cap R^o\neq\emptyset$. Let  $b\in K\cap R^o$. Then 
		\begin{enumerate}
	\item $\{b(I_m:K)\}_{m\in\NN}$ is a graded family of ideals in $R$.
	\item $\{b(I_m:K)^{\sat}\}_{m\in\NN}$ is a graded family of ideals  in $R$.  Let $bx\in b(I_m:K)^{\sat}$ and $by\in b(I_n:K)^{\sat}$ for some $x\in (I_m:K)^{\sat}$ and $y\in (I_n:K)^{\sat}$. Then $xK\mm^r\in I_m$ and $yK\mm^s\in I_n$ for some $r,s\in\NN_{>0}$. Since $b\in K$, we have $bxyK\mm^{r+s}=bx\mm^ryK\mm^s\subseteq I_mI_n\subseteq I_{mn}$ and hence $bxy\in (I_{m+n}:K)^{\sat}$ and $b^2xy\in b(I_{m+n}:K)^{\sat}$ . 
	\item If $\mathcal I$ is a filtration of ideals in $R$ then $\{b(I_m:K)\}_{m\in\NN}$ and $\{b(I_m:K)^{\sat}\}_{m\in\NN}$ are filtrations of ideals in $R$.
			\end{enumerate}	
		\end{enumerate}
	}
\end{Remark}
\begin{Remark}
	{\em Let $\{J_m\}_{m\in\NN}$ be a weakly graded family  of ideals in $R$. Consider the set $S=\{c \in R^o : cJ_mJ_n\subset J_{m+n} \mbox{ for all } m,n\geq 1 \}$. Then the graded family $\{cJ_m\}_{m\in\NN}$ is not  necessarily a Noetherian graded family for any $c\in S$.
	\\Let $K[x,y]$ be a polynomial ring  over a field $K$ and $R=K[x,y]_{(x,y)}$. Consider the weakly graded family of ideals  $\{J_m=(( x^2,xy) ^{m+1}:(x))=( x^{2m+1-i}y^i:0\leq i\leq m+1)R \}_{m\in\NN}$ and $g\in S$ be any element. Then the graded family $\mathcal P=\{P_m=gJ_m\}_{m\in\NN}$ is not Noetherian. Suppose the Rees algebra $R[\mathcal P]$ is a finitely generated $R$-algebra and generated by elements of degree up to $l$ for some $l\geq 1$. Note that $gx^{2(l+1)+1}\in R[\mathcal P]_{l+1}$. Then there exists elements $m_i\in J_{t_i}$ for some $1\leq i\leq r$ with $1\leq t_i\leq l$ such that $gx^{2(l+1)+1}=\sum r_\alpha(gm_{1})^{\alpha_1}\cdots (gm_{r})^{\alpha_r}$ for some $r_\alpha\in R$ with $\sum\limits_{i=1}^rt_i\alpha_i=l+1$. Since $\sum\limits_{i=1}^r\alpha_i\geq 2$, we get, \begin{equation}\label{2eq}x^{2(l+1)+1}=\sum r_{\alpha}g^{c_\alpha}m_{1}^{\alpha_1}\cdots m_{r}^{\alpha_r}\end{equation} where $c_\alpha=\sum\limits_{i=1}^r\alpha_i-1$.
 Hence $m_{i}\in (x^{2t_i+1}) R$ and $y\nmid m_i$ for all $1\leq i\leq r$.  Thus $\sum\limits_{i=1}^r\alpha_i\geq 2$ implies $\sum\limits_{i=1}^r(2\alpha_it_i+\alpha_i)=2(l+1) +\sum\limits_{i=1}^r\alpha_i>2(l+1)+1$  which  contradicts equation (\ref{2eq}).} 
\end{Remark}	
	\begin{Theorem}\label{Main}
		Let $(R,\mm)$ be an analytically unramified  local ring of dimension $d\geq 1$ and $K$ be any $\mm$-primary ideal in $R$. 
		\begin{enumerate}
		\item[$(1)$]  Let $\mathcal I=\{I_m\}_{m\in\NN}$ be a filtration of ideals and $\mathcal I$ satisfy $\mathcal A(r)$ condition for some $r\in\ZZ_{>0}$. Then 
					\begin{enumerate} 
		\item[$(i)$] $\epsilon(\mathcal I)=\displaystyle d!\lim\limits_{n\to \infty}\frac{\lambda\big(H^0_{\mm}(R/(I_n:K))\big)}{n^d}.$\vspace{2mm}
		\item [$(ii)$] $\epsilon(\mathcal F(n))$ is a limit  where 
		$\mathcal F(n)=\{F(n)_m=(I_{nm}:K^m) \}_{m\in\NN}$ is a graded family of ideals for all $n\in\ZZ_{>0}$ and \vspace{1mm} $$\epsilon(\mathcal I)=\displaystyle\lim\limits_{n\to\infty}\frac{\epsilon(\mathcal F(n))}{n^d}.$$
	\end{enumerate} 
		\item[$(2)$] Let $\mathcal I=\{I^m\}_{m\in\NN}$ for some ideal $I$ in $R$. For each $n\in\ZZ_{>0}$, let $J_n$ be an ideal in $R$ such that $I^n\subseteq J_n\subseteq (I^n:K)$.
		Then
	$$\epsilon (I)=\displaystyle\lim\limits_{n\to \infty}\frac{\epsilon (J_n)}{n^d}.$$
		\end{enumerate}
	\end{Theorem}	
	\begin{proof}
		Let $b\in R^o\cap K$. Then $\{b(I_m:K)\}_{m\in\NN}$ and $\{b(I_m:K)^{\sat}\}_{m\in\NN}$ are filtrations of ideals in $R$.  Since $\hat{R}$ is a faithfully flat extension of $R$, we can replace $R$, $I_n$ and $K$ by $\hat{R}$, $I^n\hat{R}$ and $K\hat{R}$ respectively. Note that $\hat{R}$ is reduced. Hence by Remark \ref{nzd1}, $b$ is a nonzerodivisor.	By  Artin-Rees lemma, there exists a positive integer $k$ such that for all $n\geq k$,
		\begin{equation}\label{eq1}bR\cap \mm^n=\mm^{n-k}(bR\cap \mm^k)\subset b\mm^{n-k}.\end{equation}		
	$(i)$	For all integers $n\geq 1$, consider the  families of ideals
		\begin{enumerate} 
			\item $\mathcal L=\{L_0=R, L_m=b^2I_m\mbox{ for all }m\geq 1\}$,	
			\item $\mathcal L(n)=\{L(n)_0=R, L(n)_1=b^2I_n, L(n)_m=b^{m+1}(I_{nm}:K) \mbox{ for all }m\geq 2\}$,	
			\item $\mathcal T=\{T_0=R, T_m=b^2I_m^{\sat}\mbox{ for all }m\geq 1\}$,
			\item $\mathcal T(n)=\{T(n)_0=R,  T(n)_m=b^{m+1}(I_{nm}:K)^{\sat}=b^{m+1}I_{nm}^{\sat}\mbox{ for all }m\geq 1\}$.
			\end{enumerate}
			Since $\mathcal I$ is a filtration of ideals, we have, $\mathcal L$, $\mathcal T$ are filtrations of ideals. 
			We show that $\mathcal L(n)$ and $\mathcal T(n)$ are graded families of ideals for all $n\geq 1$. For all $m, m'\geq 1$, using Remark \ref{nzd} $(2)(a)$ and $(2)(b)$, we get,
			\begin{eqnarray*}L(n)_mL(n)_{m'}\subseteq b^{m+1}(I_{nm}:K)b^{m'+1}(I_{nm'}:K)&\subseteq& b^{m+m'}b(I_{nm}:K)b(I_{nm'}:K)\\ &\subseteq&  b^{m+m'}b(I_{n(m+m')}:K)=L(n)_{m+m'},\end{eqnarray*}
			(since $L(n)_1=b^2I_n$, we have $L(n)_mL(n)_{m'}\subseteq b^{m+1}(I_{nm}:K)b^{m'+1}(I_{nm'}:K)$ if either $m=1$ or $m'=1$, otherwise $L(n)_mL(n)_m'= b^{m+1}(I_{nm}:K)b^{m'+1}(I_{nm'}:K)$) and 
			\begin{eqnarray*}T(n)_mT(n)_{m'}=b^{m+1}(I_{nm}:K)^{\sat}b^{m'+1}(I_{nm'}:K)^{\sat}&\subseteq& b^{m+m'}b(I_{nm}:K)^{\sat}b(I_{nm'}:K)^{\sat}\\ &\subseteq&  b^{m+m'}b(I_{n(m+m')}:K)^{\sat}=T(n)_{m+m'}.\end{eqnarray*}
		 For all integers $n\geq 1$ and $m\geq 2$, we have
		 \begin{enumerate}
		 	\item $L(n)_1=L_n$, $T(n)_1=T_n$,
		 	\item $L(n)_m=b^{m+1}(I_{nm}:K)=b^mb(I_{nm}:K)\subseteq b^mI_{nm}\subseteq b^2I_{nm}=L_{nm}$.
		 	\item $T(n)_m=b^{m+1}I_{nm}^{\sat}\subseteq b^2I_{nm}^{\sat}=T_{nm}.$
		 \end{enumerate}	
Since $\{I_m\}_{m\in\NN}$  satisfies $\mathcal A(r)$ condition for some $r\in\ZZ_{>0}$, we have $I_m^{\sat}\cap \mm^{rm}=I_m\cap \mm^{rm}$ for all $m\geq 1$. Let $c=r+2k$. We show that for all integers $m,n\geq 1$, we have $$L(n)_m\cap\mm^{cmn}=T(n)_m\cap \mm^{cmn}.$$ By definition of $\mathcal L(n)$ and $\mathcal T(n)$, we have $L(n)_m\cap\mm^{cmn}\subseteq T(n)_m\cap \mm^{cmn}$ for all $m,n\geq 1$. Let $b^{m+1}a\in  T(n)_m\cap \mm^{cmn}$ for some $a\in I_{nm}^{\sat}$. Since $b$ is a nonzerodivisor, we get, $$b^{m+1}a\in \mm^{cmn}\cap bR\subseteq b\mm^{cmn-k} \mbox{ and } b^ma\in \mm^{cmn-k}.$$ As $cmn=rmn+2kmn$ and $b$ is a nonzerodivisor, by repeated application of the Artin-Rees lemma, we obtain $a\in \mm^{rnm}$. Thus $a\in \mm^{rnm}\cap I_{nm}^{\sat}\subseteq I_{nm}$.
Hence for $m=1$, we get $b^2a\in L(n)_1\cap\mm^{cn}$  and for all  $m\geq 2$, we have, $$b^{m+1}a\in b^{m+1}I_{nm}\cap\mm^{cmn}\subseteq b^{m+1}(I_{nm}:K)\cap\mm^{cmn}= L(n)_m\cap\mm^{cmn}.$$ 
Therefore by \cite[Theorem 4.1]{CL2024}, both the limits $\displaystyle \lim\limits_{n\to\infty}\frac{\lambda_R(T_n/L_n)}{n^d}$, $\displaystyle\lim\limits_{m\to\infty}{\frac{\lambda_R(T(n)_m/L(n)_m)}{m^d}}$ exist and  $$\displaystyle\lim\limits_{n\to\infty}\lambda_R(T_n/L_n)/n^d=\lim\limits_{n\to\infty}\frac{1}{n^d}\Big({\lim\limits_{m\to\infty}{\frac{\lambda_R(T(n)_m/L(n)_m)}{m^d}}}\Big).$$ Since by \cite[Theorem 4.4]{Pa2025}, $\displaystyle\lim\limits_{n\to \infty}\frac{\lambda_R\big(H^0_{\mm}(R/(I_n:K))\big)}{n^d}={\lim\limits_{n\to\infty}{\frac{\lambda_R((I_{n}:K)^{\sat}/(I_{n}:K))}{n^d}}}$ exists, we get
		\begin{eqnarray*}\displaystyle
	\epsilon(\mathcal I)=d!	\lim\limits_{n\to\infty}{\frac{\lambda_R(	I_n^{\sat}/	I_n)}{n^d}}&=&d!\lim\limits_{n\to\infty}{\frac{\lambda_R(	b^2I_n^{\sat}/	b^2I_n)}{n^d}}= d!\lim\limits_{n\to\infty}{\frac{\lambda_R(	T_n/L_n)}{n^d}}\\&=&d!\lim\limits_{n\to\infty}\frac{1}{n^d}\Big({\lim\limits_{m\to\infty}{\frac{\lambda_R(T(n)_m/L(n)_m)}{m^d}}}\Big)\\&=&d!\lim\limits_{n\to\infty}\frac{1}{n^d}\Big({\lim\limits_{m\to\infty}{\frac{\lambda_R(b^{m+1}I_{nm}^{\sat}/b^{m+1}(I_{nm}:K))}{m^d}}}\Big)\\&=&d!\lim\limits_{n\to\infty}\frac{1}{n^d}\Big({\lim\limits_{m\to\infty}{\frac{\lambda_R(I_{nm}^{\sat}/(I_{nm}:K))}{m^d}}}\Big)\\&=&d!\lim\limits_{n\to\infty}\frac{n^d}{n^d}\Big({\lim\limits_{m\to\infty}{\frac{\lambda_R((I_{nm}:K)^{\sat}/(I_{nm}:K))}{(nm)^d}}}\Big)\\&=&
\displaystyle d!\lim\limits_{n\to \infty}\frac{\lambda_R\big(H^0_{\mm}(R/(I_n:K))\big)}{n^d},
	\end{eqnarray*} where the second last equality follows from Remark \ref{nzd} $(1)$.  
	\\$(ii)$  Consider the families of ideals
	\begin{enumerate} 
		\item $\mathcal G=\{G_0=R, G_m=b^2(I_m:K)\mbox{ for all }m\geq 1\}$,	
		\item $\mathcal G(n)=\{G(n)_0=R,  G(n)_m=b^{m+1}F(n)_m \mbox{ for all }m\geq 1\}$.
\end{enumerate} Since $\mathcal I$ is a filtration, by Remark \ref{nzd} $(2)(a)$, $\mathcal G$ is a filtration of ideals and since $\mathcal F(n)$ is a graded family of ideal, we have $\mathcal G(n)$ is a graded family of ideals for all $n\geq 1$. 	
\\For all integers $n\geq 1$, we have
	$G(n)_1=G_n$. For any integers  $n\geq 1$ and $m\geq 2$, let  $b^{m+1}a\in G(n)_m$ for some $a\in F(n)_m= (I_{nm}:K^m)$. Since $b\in K$, we have $b^{m-1}aK\subseteq aK^m\subseteq I_{nm}$. Therefore $b^{m-1}a\in(I_{nm}:K)$ and $b^{m+1}a=b^2b^{m-1}a\in b^2(I_{nm}:K)=G_{nm}$. Hence $G(n)_m\subseteq G_{nm}$ for all $n\geq 1$ and $m\geq 2$.
	\\ The graded families $\mathcal T$, $\mathcal L(n)$ and $\mathcal T(n)$ are the same graded families defined in part $(i)$.
	Since $L(n)_m\subseteq G(n)_m$ for all $n,m\geq 1$, for the same $c$ as in part $(i)$, we have  \begin{equation}\label{eq2} G(n)_m\cap\mm^{cmn}=T(n)_m\cap \mm^{cmn}\mbox{ for all integers }m,n\geq 1.\end{equation} 
	Therefore by \cite[Theorem 4.1]{CL2024}, we get, both the limits \\$\displaystyle \lim\limits_{n\to\infty}\frac{\lambda_R(T_n/G_n)}{n^d}, \displaystyle\lim\limits_{m\to\infty}{\frac{\lambda_R(T(n)_m/G(n)_m)}{m^d}}$ exist and 
	$$ \displaystyle\lim\limits_{n\to\infty}\lambda_R(T_n/G_n)/n^d =\lim\limits_{n\to\infty}\frac{1}{n^d}\Big(\lim\limits_{m\to\infty}{\frac{\lambda_R(T(n)_m/G(n)_m)}{m^d}}\Big).$$
	Now we show that $\epsilon (\mathcal F(n))$ is a limit for all $n\geq 1$. Fix $n\geq 1$. By Remark \ref{nzd} $(1)$, for all $m\geq 1$, we have \begin{equation}\label{eq4}b^{m+1}F(n)_m^{\sat}=b^{m+1}(I_{nm}:K^m)^{\sat}=b^{m+1}I_{nm}^{\sat}=T(n)_m.\end{equation} Thus for all $m\geq 1$ and for any $x\in F(n)_m^{\sat}\cap\mm^{cmn}$, using equation (\ref{eq2}), we have $$b^{m+1}x\in b^{m+1}F(n)_m^{\sat}\cap\mm^{cmn}= T(n)_m\cap\mm^{cmn}=G(n)_m\cap\mm^{cmn}\subseteq G(n)_m=b^{m+1}F(n)_m.$$ Since $b$ is a nonzerodivisor, we have $x\in F(n)_m$. Therefore, for all $m\geq 1$ and $c'=cn$, we get,
	$$F(n)_m\cap \mm^{c'm}=F(n)_m^{\sat}\cap \mm^{c'm}.$$ 
	Suppose $F(n)_1\subseteq P$ for some $P\in\Min(R)$ with $\dim (R/P)=d$. We show that $F(n)_m\subseteq P$ for all $m\in\ZZ_{>0}$.
		Let $x\in F(n)_m=(I_{nm}:K^m)$. Then  $xK^{m-1}\subseteq (I_{nm}:K)\subseteq (I_n:K)=F(n)_1\subseteq P$. Since $d\geq 1$ and $K$ is $\mm$-primary, we get $x\in P$. Thus by \cite[Theorem 6.1]{C1}, we have, $\epsilon(\mathcal F(n))$ is a limit. 
	 \\Using equation (\ref{eq4}), we obtain,
	 \begin{eqnarray*}
	 {\lim\limits_{m\to\infty}d!{\frac{\lambda_R(T(n)_m/G(n)_m)}{m^d}}}&=&\lim\limits_{m\to\infty}d!{\frac{\lambda_R(b^{m+1}(I_{nm}:K^m)^{\sat}/b^{m+1}(I_{nm}:K^m) )}{m^d}}\\&=& \lim\limits_{m\to\infty}d!{\frac{\lambda_R(b^{m+1}F(n)_m^{\sat}/b^{m+1}F(n)_m )}{m^d}}
	 \\&=&\lim\limits_{m\to\infty}d!{\frac{\lambda_R(F(n)_m^{\sat}/F(n)_m )}{m^d}}=\epsilon(\mathcal F(n)).
	 	\end{eqnarray*} Hence by part $(i)$, we get 
	 \begin{eqnarray*}\epsilon(\mathcal I)=\displaystyle
	 		\lim\limits_{n\to \infty}d!\frac{\lambda\big(H^0_{\mm}(R/(I_n:K))\big)}{n^d}&=&\displaystyle{\lim\limits_{n\to\infty}d!{\frac{\lambda_R(b^2I_{n}^{\sat}/b^2(I_{n}:K))}{n^d}}}=	\lim\limits_{n\to\infty}d!\frac{\lambda_R(T_n/G_n)}{n^d}\\&=&\displaystyle\lim\limits_{n\to\infty}\frac{1}{n^d}\big({\lim\limits_{m\to\infty}d!{\frac{\lambda_R(T(n)_m/G(n)_m)}{m^d}}}\big)=\lim\limits_{n\to\infty}\frac{\epsilon(\mathcal F(n))}{n^d}.
	 	\end{eqnarray*}
$(2)$ Since $I^n\subseteq J_n\subseteq (I^n:K)$ and $K$ is an $\mm$-primary ideal , we have $$I^{nm}\subseteq J_n^m\subseteq (I^n:K)^m\subseteq F(n)_m:= (I^{nm}:K^m)\subseteq (I^{nm}:K^m)^{\sat}=(I^{nm})^{\sat}$$ for all $n, m\in\ZZ_{>0}$. Hence $$(J_n^m)^{\sat}= ((I^n:K)^m)^{\sat}= (I^{nm}:K^m)^{\sat}=(I^{nm})^{\sat}$$ for all $n, m\in\ZZ_{>0}$. Thus for all $n\geq 1$, we get \begin{eqnarray}\label{eqeq}\displaystyle\frac{\epsilon(J_n)}{n^d}=\frac{1}{n^d}\Big(\lim\limits_{m\to\infty}\frac{\lambda_R((J_n^m)^{\sat}/J_n^m)}{m^d}\Big)\nonumber&\leq& \frac{1}{n^d}\Big(\lim\limits_{m\to\infty}\frac{\lambda_R((I^{nm})^{\sat}/I^{nm})}{m^d}\Big)\\&=&\lim\limits_{m\to\infty}\frac{\lambda_R((I^{nm})^{\sat}/I^{nm})}{(nm)^d}=\epsilon(I).\end{eqnarray}	
Let $\mathcal F(n)=\{F(n)_m= (I^{nm}:K^m)\}_{m\in\NN}$ for all $n\in\ZZ_{>0}$. Therefore by part $(1)$, we get 
\begin{eqnarray*}
	\displaystyle\epsilon(I)=\lim\limits_{n\to\infty}\frac{\epsilon(\mathcal F(n))}{n^d}\leq\liminf_{n\to\infty} \frac{\epsilon(J_n)}{n^d}\leq \limsup_{n\to\infty} \frac{\epsilon(J_n)}{n^d}\leq \epsilon(I).
\end{eqnarray*}	
		\end{proof}
		\begin{Example} {\label{ex1}}{\em We give an example of a Noetherian filtration $\mathcal I$ such that all the graded families $\mathcal F(n)$ are non-Notherian for all $n\in\ZZ_{>0}$.
				\\ Consider the subring $k[x,y^2,y^3]$ of a polynomial ring $k[x,y]$ over a field $k$. Let $R=k[x,y^2,y^3]_{(x,y^2,y^3)}$, $\mm=(x,y^2,y^3)R$. 
				\\Take the ideal  $I=(x^2,xy^2)R$ and   the filtartion $\mathcal I=\{I^m\}_{m\in\NN}$. By \cite[Theorem 3.4]{Sw}, $\mathcal I$ satisfies $\mathcal A(r)$ for some $r\in\ZZ_{>0}$.
				Let $J=(x,y^2)$. Then $I=xJ$. Since $\mm^2=J\mm$, by induction, we get   $\mm^t=J^{t-1}\mm$ for all $t\geq 2$. 
				\\Now for all $m\geq 1$,  $I^m\subseteq x^mR$ implies $(I^m)^{\sat}\subseteq x^mR$ and $x^m\mm^{m+1}=x^mJ^m\mm\subseteq I^m$ implies $x^m\in (I^m:\mm^{m+1})\subseteq (I^m)^{\sat}$. Therefore  $(I^m)^{\sat}=x^mR$ for all $m\geq 1$. 
				\\Since $x,y^2$ is a regular sequence in $R$, we have $J^p/J^{p+1}\simeq (R/J)^{p+1}$ for all $p\geq 1$, $\lambda_R(R/J)=2$ 
				and $\lambda_R(J^p/J^{p+1})=2(p+1)$. Since $$\lambda_R((I^m)^{\sat}/I^m)=\lambda_R(R/J^m)=\sum\limits_{p=0}^{m-1}\lambda_R(J^p/J^{p+1})=\sum\limits_{p=0}^{m-1}2(p+1)=m(m+1),$$ we have 
				$$\displaystyle\epsilon(I)= 2\lim\limits_{m\rightarrow \infty}\frac{\lambda_R(H^0_{\mm}(R/I^m))}{m^2}=2.$$
				Now we show that $(I^{nm}:\mm^m)=x^{nm}J^{nm-m}\mm$ for all $n,m\in\ZZ_{>0}$.
				For all $m,n\geq 1$, we have, $$x^{nm}J^{nm-m}\mm^{m+1}=x^{nm}J^{nm-m}J^m\mm=I^{nm}\mm\subseteq I^{nm}\mbox{ and hence }x^{nm}J^{nm-m}\mm\subseteq (I^{nm}:\mm^m).$$ Now we show the converse. It is enough to show the following inclusion for the monomial ideals $((x^2,xy^2)^{nm}:(x,y^2,y^3)^m)=x^{mn}(x,y^2)^{nm-m}(x,y^2,y^3)$ in $k[x,y^2,y^3]$. We use the same notation $I$, $J$ and $\mm$ for the above monomial ideals.
				Let $m,n\geq 1$ and $a\in (I^{nm}:\mm^m)$ be a monomial. Then $a(y^2)^{m}\in I^{nm}=x^{nm}J^{nm}$ implies $a\in (x^{nm})$ as $x,y^2$ is a regular sequence. Therefore $a=vx^{nm}$ for some monomial $v$ in $k[x,y^2,y^3]$ and $vx^{nm}\mm^m\subseteq x^{nm}J^{nm}$ implies $v\mm^m\subseteq J^{nm}$. We show that $v\in J^{nm-m}\mm$.
				\\ As $(y^3)^2=(y^2)^3$, every monomial in $k[x,y^2,y^3]$ is either of the form $x^p(y^2)^q$ or $x^p(y^2)^qy^3$ for some $p,q\in\NN$. If $v=x^p(y^2)^q$ then $vy^{2m+1}=v(y^2)^{m-1}y^3\in v\mm^m\subseteq J^{nm}$ implies $p+q\geq nm-m+1$ and $v\in J^{nm-m+1}\subseteq J^{nm-m}\mm$ \\(Suppose $p+q<nm-m+1$. Since  $v(y^2)^{m-1}y^3\in J^{nm}$, we have $v(y^2)^{m-1}y^3=x^{nm-i}(y^2)^iu$ for some $0\leq i\leq nm$ and for some monomial $u$ in $k[x,y^2,y^3]$. Thus $nm-i\leq p$. Then $q+m-1<nm-p\leq i$ and hence $x^{p-nm-i}y^3=(y^2)^{i-(q+m-1)}u$ which is a contraditction as $y\notin k[x,y^2,y^3]$). \vspace{1mm}
			\\	If $v=x^p(y^2)^qy^3$ then $v(y^2)^{m}\in v\mm^m\subseteq J^{nm}$ implies $p+q+m\geq nm$ and hence $p+q\geq nm-m$. Therefore $v=x^p(y^2)^qy^3\in J^{nm-m}\mm$. Thus $a\in x^{nm}J^{nm-m}\mm$. 
				\vspace{1mm}
				\\Since $x^{2mn-m+1}$ is a minimal generator of $(I^{nm}:\mm^m)$ for all $n,m\in\ZZ_{>0}$, the graded families $\mathcal F(n)=\{(I^{nm}:\mm^m)\}_{m\in\NN}$ are non-Noetherian for all $n\in\ZZ_{>0}$.\vspace{2mm} 
				\\Now for all $n,m\in\ZZ_{>0}$, $(I^{nm}:\mm^n)=x^{nm}J^{nm-m}\mm=x^{nm}\mm^{nm-m+1}$ (for $n=1$, it is clear and for $n\geq 2$, use $\mm^n=\mm J^{n-1}$). We have $(I^{nm}:\mm^m)^{\sat}=(I^{nm})^{\sat}=(x^{nm})$. Therefore for all $n,m\in\ZZ_{>0}$, $$\lambda_R(H_\mm^0(R/(I^{nm}:\mm^n)))=\lambda_R((x^{nm})/x^{nm}\mm^{nm-m+1})=\lambda_R(R/\mm^{nm-m+1}).$$
				Since $\mm^p=J^p+y^3J^{p-1}$ for all $p\geq 1$, $\mm^p/\mm^{p+1}$ is a $R/\mm$-vector space of dimension $2p+1$.  
				Therefore $$\lambda_R(R/\mm^{nm-m+1})=\sum\limits_{p=0}^{nm-m}\lambda_R(\mm^p/\mm^{p+1})=\sum\limits_{p=0}^{nm-m} (2p+1)=(nm-m+1)^2.$$ Hence for all $n\geq 1$, 
				$$\epsilon(\mathcal F(n))=2\lim\limits_{m\rightarrow \infty}\frac{\lambda_R(H^0_{\mm}(R/F(n)_m))}{m^2}=2\lim\limits_{m\rightarrow \infty}\frac{\lambda_R(H^0_{\mm}(R/(I^{nm}:\mm^m)))}{m^2}=2(n-1)^2 \mbox{ and }$$
				$$\displaystyle\epsilon(I)=2=\lim\limits_{n\to\infty}\frac{\epsilon(\mathcal F(n))}{n^2}.$$	}
		\end{Example}	
	We conclude this section by investigating the maximality of the analytic spread of filtrations of ideals.  As a consequence of Theorem \ref{Main}, we get the following Corollary.
		\begin{Corollary}{\label{cor1}}
			Let $(R,\mm)$ be an analytically unramified  local ring of dimension $d\geq 1$, $I$ be an ideal in $R$ and $K$ be any $\mm$-primary ideal in $R$. For each $n\in\ZZ_{>0}$, let $J_n$ be an ideal in $R$ such that $I^n\subseteq J_n\subseteq (I^n:K)$. Then the following are equivalent.
			\begin{enumerate}
				\item[$(i)$] $\ell(I)=d$.
				\item[$(ii)$]  $\ell(J_n)=d$ for all $n\gg0$. 
				\item[$(iii)$]  $\ell(J_n)=d$ for some $n\geq 1$.
			\end{enumerate} 
		\end{Corollary}		
		\begin{proof} By \cite[Theorem 4.7]{KV}, $\ell(I)=d$ if and only if $\epsilon(I)>0$. 
			\\$(i)\Rightarrow (ii)$ We have $\epsilon(I)>0$. Since $\epsilon(J_n)\geq 0$ for all $n\geq 1$, by Theorem \ref{Main} $(2)$, we get   $\epsilon(J_n)> 0$  for all $n\gg0$ and hence $\ell(J_n)=d$ for all $n\gg0$. $(ii)\Rightarrow (iii)$ is clear.  $(iii)\Rightarrow (i)$ The result follows from equation (\ref{eqeq}).
		\end{proof}	
\begin{Theorem}
	Let $(R,\mm)$ be an analytically unramified  local ring of dimension $d\geq 1$ and $K$ be any $\mm$-primary ideal in $R$. Let $\mathcal I=\{I_m\}_{m\in\NN}$ be a filtration of ideals in $R$. Consider the graded families of ideals $\mathcal F(n)=\{ F(n)_m=(I_{nm}:K^m) \}_{m\in\NN}$  for all integers $n\geq 1$.
	\begin{enumerate}
		\item[$(1)$] Supppose $\mathcal I$ satisfies $\mathcal A(r)$ condition for some $r\in\ZZ_{>0}$. Then the following are equivalent.	\begin{enumerate}
		\item[$(i)$] $\epsilon(\mathcal I)>0$,
		\item[$(ii)$] $\epsilon(\mathcal F(n))>0$ for all $n\gg0$,
		\item[$(iii)$] $\epsilon(\mathcal F(n))>0$ for some $n\geq 1$.
	\end{enumerate}
	\item[$(2)$] If $R$ is an
	excellent normal local domain of dimension $d\geq 1$ of equicharacteristic zero, or an arbitrary excellent local domain with $d\leq 3$ and $\mathcal I$ is a $\QQ$-divisorial filtration of ideals in $R$ then  the following are equivalent.
	\begin{enumerate}
		\item[$(i)$] $\ell(\mathcal I)=d$,
		\item[$(ii)$] $\ell(\mathcal F(n))=d$ for all $n\gg0$,
		\item[$(iii)$] $\ell(\mathcal F(n))=d$ for some $n\geq 1$. 
		\end{enumerate}	
\end{enumerate}	
\end{Theorem}	
\begin{proof} 
$(1)$	$(i)\Rightarrow (ii)$ Since $\epsilon(\mathcal F(n))\geq 0$ for all $n\geq 1$, by Theorem \ref{Main} $(1)(ii)$, we get the required result.  $(ii)\Rightarrow (iii)$ is clear. For $(iii)\Rightarrow (i)$, note that, $ I_{nm}\subseteq(I_{nm}:K^m)$ and $(I_{nm}:K^m)^{\sat}=I_{nm}^{\sat}$ for all $n,m\geq 1$. Hence $\displaystyle\lambda_R\big(\frac{(I_{nm}:K^m)^{\sat}}{(I_{nm}:K^m)}\big)\leq \lambda_R\big(\frac{I_{nm}^{\sat}}{I_{nm}}\big)$ for all $n,m\geq 1$. Therefore for all $n\geq 1$, $$\displaystyle\epsilon(\mathcal F(n))=\lim\limits_{m\to \infty}\frac{\lambda_R\big({(I_{nm}:K^m)^{\sat}}/{(I_{nm}:K^m)}\big)}{m^d}\leq n^d\epsilon(\mathcal I).$$ Hence the result follows. 
	\\$(2)$ Now suppose $R$ is an
	excellent normal local domain of dimension $d\geq 1$ of equicharacteristic zero, or an arbitrary excellent local domain with $d\leq 3$ and $\mathcal I$ is a $\QQ$-divisorial filtration of ideals in $R$. By \cite[Theorem 3.1]{CS2024}, $\mathcal I$ satisfies $\mathcal A(r)$ condition for some $r\in\ZZ_{>0}$. 
	\\Let $\nu$ be an $\mm$-divisorial valuation of $R$ and $\nu(K):=\min\{\nu(r):r\in K\}=b$.  Thus $\nu(K^m)=mb$ for all $m\geq 1$. Let $z_m\in K^m$ such that $\nu(z_m)=mb$ for all $m\geq 1$. We show that $(I(\nu)_{\lceil nma\rceil}:K^m)=I(\nu)_{\lceil (na-b)m\rceil}$ for all $n,m\geq 1$ and any $a\in\QQ_{>0}$ and hence $\{ (I(\nu)_{\lceil nma\rceil}:K^m)=I(\nu)_{\lceil (na-b)m\rceil}\}_{m\in\NN}$ is a filtration of ideals. 
	\\Since $\nu$ is a discrete valuation, we have $b\in \ZZ$ and  \begin{eqnarray*}
		x\in (I(\nu)_{\lceil nma\rceil}:K^{m})\Leftrightarrow\nu(xz_{m})\geq {\lceil nma\rceil}&\Leftrightarrow& \nu(x)+mb\geq {\lceil nma\rceil}\\&\Leftrightarrow&\nu(x)\geq {\lceil nma-mb\rceil}\Leftrightarrow x\in I(\nu)_{\lceil (na-b)m\rceil}.
	\end{eqnarray*}
	Let $\mathcal I=\{I_m=I(\nu_1)_{\lceil ma_1\rceil}\cap\cdots\cap I(\nu_r)_{\lceil ma_r\rceil}\}_{m\in\NN}$ with $a_i\in\QQ_{>0}$ for all $1\leq i\leq r$, $\mm_{\nu_i}\cap R=\mm$ for all $1\leq i\leq t$  and $\mm_{\nu_i}\cap R\subsetneq \mm$ for all $t+1\leq i\leq r$. Let $\nu_i(K)=b_i$ for all $1\leq i\leq t$. Thus $\nu_i(K^m)=mb_i$ for all $1\leq i\leq t$ and $m\geq 1$. Since $K$ is an $\mm$-primary ideal, we have $(I(\nu_i)_{\lceil nma_i\rceil}:K^m)=I(\nu_i)_{\lceil nma_i\rceil}$ for all $t+1\leq i\leq r$.
\\Let $\mathcal C_n=\{j:1\leq j\leq r\mbox{ and } na_j-b_j>0\}$ and $\mathcal D_n=\{1,\ldots,r\}\setminus \mathcal C_n$ for all $n\geq 1$. Then for all $n\geq 1$, we have 
\begin{equation}{\label{10eq}}\mathcal F(n)=\{F(n)_m=(I_{nm}:K^m)=(\bigcap\limits_{i\in\mathcal C_n}I(\nu_i)_{\lceil (na_i-b_i)m\rceil})\bigcap ({\bigcap\limits_{t+1\leq i\leq r}I(\nu_i)_{\lceil mna_i\rceil}})\}_{m\in\NN}\end{equation} is a  filtration of ideals.
\\$(i)\Rightarrow (ii)$ Suppose $\ell(\mathcal I)=d$. Then by \cite[Theorem 4.1]{CS2024}, $\epsilon(\mathcal I)>0$  and hence $\epsilon(\mathcal F(n))>0$ for all $n\gg0$ by part $(1)$. Therefore $\mm\in\Ass(F(n)_m)$ for all $n,m\gg 0$ 
and hence $\mathcal C_n\neq\emptyset$ for all $n\gg 0$. Thus $\mathcal F(n)$ is a $\QQ$-divisorial filtration of ideals for all $n\gg 0$ and by \cite[Theorem 4.1]{CS2024}, we get $\ell(\mathcal F(n))=d$ for all $n\gg0$. 
	\\The implication $(ii)\Rightarrow (iii)$ is clear. 
	\\$(ii)\Rightarrow (i)$
	Suppose  $\ell(\mathcal F(n))=d$ for some $n\geq 1$, say for $n=n_0$. If $t=r$ then $\mathcal I$ is a filtration of $\mm$-primary ideals and hence by \cite[Theorem 1.5]{C2025}, $\ell(\mathcal I)=d$. Now suppose $t<r$. By the expression (\ref{10eq}), we get, $\mathcal F(n_0)$
	is a $\QQ$-divisorial filtration. Therefore  by \cite[Theorem 4.1]{CS2024}, $\epsilon(\mathcal F(n_0))>0$ and hence by part $(1)$, we have $\epsilon(\mathcal I)>0$. Thus again using \cite[Theorem 4.1]{CS2024}, we get $\ell(\mathcal I)=d$. 
		\end{proof}	
		\section*{Acknowledgments}
		The author was partially supported by ARG MATRICS Grant with Grant No.
		\\ ANRF/ARGM/2025/000862/MTR
		

\begin{thebibliography}{1000000000}
			
			\bibitem{CHST05} S. D. Cutkosky, H. T. H\`{a}, H. Srinivasan, and E. Theodorescu, Asymptotic behavior of the length of local cohomology, Canad. J. Math. 57 (2005), 1178-1192.
			
			\bibitem{CHS}  S.D. Cutkosky, J. Herzog and Hema Srinivasan, Finite generation of algebras associated to powers of ideals,  Math.  Proc.  of the Camb. Phil.  Soc.  148 (2010), 55-72.
			
			\bibitem{C1} S.D. Cutkosky, Asymptotic multiplicities of graded families of ideals and linear series,  Advances in Math. 264 (2014), 55-113.
			
			\bibitem{C2021} S. D. Cutkosky, The Minkowski equality of filtrations, Advances in Math. 388 (2021).
			
			\bibitem{C2025} S. D. Cutkosky, The Rees algebra and analytic spread of a divisorial filtration, Advances in Math. 479 (2025).
			
				\bibitem{CL2024} S. D. Cutkosky and S. Landsittel, Epsilon multiplicity is a limit of Amao multiplicities, J. Algebra Appl. 24 (2025), no. 13-14, Paper No. 2541021, 18 pp
				
				\bibitem{CS2022} S. D. Cutkosky and  P. Sarkar, Analytic spread of filtrations and symbolic algebras, J. Lond. Math. Soc. (2) 106 (2022), no. 3, 2635-2662.
			
			\bibitem{CS} S. D. Cutkosky and  P. Sarkar,  Multiplicities and mixed multiplicities of arbitrary filtrations, Res Math Sci 9, 14 (2022).
		
		\bibitem{CS2024} S. D. Cutkosky and  P. Sarkar, Epsilon multiplicity and analytic spread of filtrations,  Illinois J. Math. 68 (2024), no. 1, 189-210.
		
		\bibitem{DM} S. Das and C. Meng, Asymptotic colength for families of ideals : an analytic approach, to appear in J. Algebra, arXiv:2410.11991. 

		
\bibitem{HPV} J. Herzog, T. J. Puthenpurakal, and J. K. Verma, Hilbert polynomials and powers of ideals, Math. Proc. Cambridge Philos. Soc. 145 (2008), 623-642.
		
		\bibitem{hilb} D. Hilbert,  Ueber die Theorie der algebraischen Formen, Math. Ann., 36, (1890), 473-534.
		
		\bibitem{JMV} J. Jeﬀries, J. Mont\~{a}no, and M. Varbaro Multiplicities of classical varieties, Proc. London Math. Soc. 110 (2015), 1033-1055.
		
		\bibitem{KV} D. Katz and J. Validashti, Multiplicities and Rees valuations, Collect. Math. 61, 1 (2010), 1-24.
		
		\bibitem{Li} J. Lipman, Equimultiplicity, reduction and blowing up,  R.N. Draper (Ed.), Commutative Algebra, Lect. Notes Pure Appl. Math., vol. 68, Marcel Dekker, New York (1982), pp. 111-147.
		
		\bibitem{Mat} H. Matsumura, Commutative Ring Theory, Cambridge University Press, 1986.
		
		\bibitem {samuel} P. Samuel,  La notion de multiplicit\'e en alg\`ebre et en 
		g\'eom\'etrie alg\'ebrique, J. Math. Pures Appl., 30, (1951), 159-274.
		
		\bibitem{Pa2025} P. Sarkar, Multiplicities of weakly graded families of ideals, Bull. Lond. Math. Soc. 57 (2025), no. 8, 2354-2371.
		
		
		\bibitem{Sw} I. Swanson, Powers of ideals, primary decompositions, Artin-Rees lemma and regularity, Math. Ann. 307 (1997), no. 2, 299-314.
		
		\bibitem{HS} I. Swanson and C. Huneke, Integral closure of ideals, rings and modules,
		Cambridge Univ. Press, 2006.
		
		
		\bibitem{UV} B. Ulrich and J. Validashti, Numerical criteria for integral dependence, Math. Proc. Camb. Phil. Soc 151 (2011), 95-102.
		
	\end{thebibliography}
	\end{document}